\documentclass[12pt]{amsart}
\usepackage[margin=2cm]{geometry}
\usepackage[utf8]{inputenc}
\usepackage[english]{babel}
\usepackage{amsthm}
\usepackage{mathtools}
\usepackage{enumerate}
\usepackage{mathrsfs}
\usepackage{tikz-cd}
\usepackage{bbm}
\usepackage[hidelinks]{hyperref}
\usepackage[nameinlink,capitalize,noabbrev]{cleveref}
\usepackage[style=alphabetic,backend=bibtex,maxalphanames=4,maxnames=10]{biblatex}
\makeatletter
\let\@afterindenttrue\@afterindentfalse
\@afterindentfalse
\makeatother

\newtheorem{theorem}[equation]{Theorem}

\newtheorem{lemma}[equation]{Lemma}
\theoremstyle{definition}
\newtheorem{definition}[equation]{Definition}

\newtheorem{remark}[equation]{Remark}
\usepackage{amsmath}
\usepackage{amssymb}
\usepackage{graphicx}
\usepackage{hyperref}
\newcommand{\N}{\mathbb{N}}
\newcommand{\Noo}{\mathbb{N}_\infty}
\newcommand{\Z}{\mathbb{Z}}

\newcommand{\id}{\mathsf{id}}

\newcommand{\colim}{\mathrm{colim}}

\newcommand{\emb}{\hookrightarrow}

\newcommand{\C}{\mathscr C}
\newcommand{\E}{\mathscr E}
\newcommand{\lcond}{\mathrm{LCond}}
\newcommand{\epito}{\twoheadrightarrow}

\DeclareMathOperator{\Hom}{\underline{Hom}}
\DeclareMathOperator{\Ab}{Ab}
\DeclareMathOperator{\Set}{Set}

\title{On internally projective sheaves of groups}
\author{David Wärn}

\begin{document}
\begin{abstract}

A sheaf of modules on a site is said to be internally projective if sheaf hom
with the module preserves epimorphism. In this note, we give an example
showing that internally projective sheaves of abelian groups 
are not in general stable under base change to a slice. This shows that internal
projectivity is weaker than projectivity in the internal logic of
the topos, as expressed for example in terms of Shulman's stack semantics. 
The sheaf of groups that we use as a counterexample comes from recent work
by Clausen and Scholze on light condensed sets.

\end{abstract}
\maketitle

\section{Introduction}
An object $P$ of a category $\C$ is said to be \emph{projective}
if for any objects $X, Y$ of $\C$ and for any epimorphism $f : X \epito Y$, postcomposition
$f \circ - : \C(P, X) \to \C(P, Y)$ with $f$ is surjective.
In this note we consider a variation of this notion: if
$\C$ comes equipped with some kind of internal hom functor
$\Hom : \mathscr C^{\mathrm{op}} \times \C \to \C$, we may say
that an object $P$ is \emph{internally projective} if for any epimorphism
$f : X \epito Y$ in $\C$, the morphism $\Hom(P, f) : \Hom(P, X) \to \Hom(P, Y)$
is again an epimorphism in $\C$.
In particular we are concerned with the case where $\E$ is a topos,
$\C$ is the category $\Ab(\E)$ of abelian group objects in $\E$,
and $\Hom$ is sheaf hom.

Any topos has a rich internal language that allows to interpret ordinary
mathematical notions inside the topos \cite{geology}. In this way the notions
of `abelian group', `group homomorphism', and `surjection', when interpreted in
a topos $\E$, give rise to the notions of `abelian group object in $\E$',
`sheaf hom', and `epimorphism', respectively. Thus one might expect, as the
author once did, that the notion of `projective abelian group', when
interpreted in a topos $\E$, gives rise to the notion of `internally
projective abelian group object'. The purpose of this note is to dispel this
expectation: in fact an internally projective abelian group object need not
be projective in the internal language of $\E$.

The difference between the two notions arises from the interpretation of 
the universal quantificafier in `\emph{for all} abelian groups $X$, $Y$ \ldots'
that appears in the notion of a projective abelian group.
The point is that, as always when interpreting a universal quantificatier in a topos,
one must consider not just the things that `exist in the current stage',
but also things that `may appear in later stages'.
Explicitly, this means that an abelian group object $P \in \Ab(\E)$ is projective in the
internal language of $\E$ if for all objects $S \in \E$ 
and all abelian group objects $X, Y \in \Ab(\E / S)$ \emph{in the slice over $S$}
with an epimorphism $f : X \epito Y$, the induced morphism
\[\Hom(S^\star P, f) : \Hom(S^\star P, X) \to \Hom(S^\star P, Y)\] is an epimorphism.
Here $\Hom$ denotes sheaf hom in $\Ab(\E / S)$, and $S^\star P \in \Ab(\E / S)$ 
is the base change of $P \in \Ab(\E)$ along the unique morphism $!_S : S \to 1$
to the terminal object.
Conflating an abelian group with its underlying set, one might
write $S \times P$ instead of $S^\star P$.

In short, an abelian group object $P$ in a topos $\E$ is projective in the
internal language if and only if for every object $S$, the base change of $P$
to $\E / S$ is internally projective. 
Following Harting~\cite{harting}, we say $P$ is \emph{invariant-internally} projective in this case.
It is clear that if $P$ is invariant-internally projective, then it is 
internally projective, since we can take $S$ to be the terminal object.

The condition that $\Hom(P, f)$ is an epimorphism for a given $f : X \to Y$ can be understood more
explicitly as follows. It means that for any object $S \in \E$ and any homomorphism
$y : S^\star P \to S^\star Y$ in $\Ab(\E / S)$, there is an epimorphism $t : T
\epito S$ and a lift $x : T^\star P \to T^\star X$ such that
$T^\star f \circ x = t^\star y$. We note that, while the notion of internal
projectivity only takes into account the \emph{global} epimorphisms of $\Ab(\E)$, it
does take into account general homomorphism $y : S^\star P \to S^\star Y$ and not just the global
homomorphisms $P \to Y$.

It remains to give an example of topos $\E$ with an abelian group object $P$
and an object $S$, such that $P$ is internally projecive in $\Ab(\E)$, but 
$S^\star P$ is not internally projective in $\Ab(\E / S)$. To this end, we consider
the topos $\lcond$ of light condensed sets recently introduced by Clausen and
Scholze~\cite{analytic}.
Central to their development is the internal projectivity of a certain free abelian 
group object $\Z[\Noo]$. Here $\Noo$ is the one-point compactification of $\N$.%
\footnote{Clausen and Scholze use the notation $\N \cup \{\infty\}$ for what we call
$\Noo$. We prefer not to use this notation since
from the internal perspective $\Noo$ is \emph{not} a union of $\N$ and $\{\infty\}$.}
The contribution of this note is simply to prove that $\Noo^\star\, \Z[\Noo]$ is not
internally projective in $\Ab(\lcond / \Noo)$, so that
$\Z[\Noo]$ is internally projective but not invariant-internally projective.

\subsection*{Internal injectivity and projective objects of a topos}
The situation changes if instead of internally projective abelian group objects
in $\E$, one considers internally \emph{injective} (abelian group) objects in
$\E$, or internally projective objects of $\E$. The basic picture remains the
same: for example, given an abelian group object $P$ of $\E$ one can ask if $P$
is internally injective and if it is invariant-internally injective. 
Harting~\cite[Theorem 1.1]{harting} showed that in fact every internally injective
abelian group object is invariant-internally injective.
Blechschmidt~\cite[Theorem 3.7]{ingo} showed that the same holds for 
$\E$ in place of $\Ab(\E)$, i.e.\ for \emph{objects} of $\E$.
The reason is the same in each case: it boils down to the fact that the left adjoints
$\bigoplus_S : \Ab(\E / S) \to \Ab(\E)$ and $S_! : \E/S \to \E$ preserve monomorphisms.
The dual statement, that the right adjoint $\Pi_S : \E / S \to \E$ preserves epimorphisms,
is not true in general topoi, but still it \emph{is} true that any internally
projective object of $\E$ is invariant-internally 
projective~\cite[Chapter IV, Exercise 16(b)]{geology}.
We explain this in \cref{object} and show why the argument fails for abelian group objects.

For the reader looking to read more on this topic, we refer to a recent article of
Christensen and Taxerås Flaten~\cite{ext}, in which the question answered in this note
was posed.

\subsection*{Quantification over objects}
In this note we have no reason to treat formally the notion of
`interpreting a sentence in the internal language of $\E$', but one
remark is in order: while the traditional Kripke--Joyal semantics for the internal
language of a topos does not support quantification over objects,
Shulman's stack semantics \cite[Definition 7.2]{shulman} does. Thus
while we work with the notion of `projective in the internal language' in an ad hoc
manner, it can be obtained more systematically.

\subsection*{Acknowledgements}
I would like to thank Thierry Coquand for helpful discussions and Peter Scholze for
elaborating on the proof that $\Z[\Noo]$ is internally projective.
This note was born out of ongoing work by Felix Cherubini, Thierry Coquand, 
Freek Geerligs and Hugo Moeneclaey on the internal logic of the topos of light condensed sets~\cite{ssd}.

\section{Light condensed sets and internal projectivity}

We recall the definition of light condensed sets introduced by
Clausen and Scholze \cite{analytic}.

\begin{definition}
A \emph{light profinite set} is a topological space that arises as a countable limit of
finite discrete sets.
The category of profinite sets and continuous maps has a \emph{coverage},
where a cover of $S$ is given by a finite collection of maps $\{p_i  : T_i \to S \mid i = 1,\ldots,n \}$
that are \emph{jointly} surjective.

The topos $\lcond$ of \emph{light condensed sets} is the category of sheaves on the above site.
\end{definition}

The category of light profinite sets admits several equivalent definitions;
for example it is dual to the category of countably presented Boolean algebras.

For our purposes, one light profinite set plays a distinguished role: the
one-point compactification $\Noo$ of the naturals, also known as the generic
convergent sequence. This can be described as the limit of the sequence
$\{\infty\} \leftarrow \{0, \infty\} \leftarrow \{0, 1, \infty\} \leftarrow \cdots$,
where the map $\{0,\cdots,n+1, \infty\} \to \{0,\cdots,n, \infty\}$
fixes every element except $n+1$ which is mapped to $\infty$.
We write $\Noo$ also for the corresponding representable object of $\lcond$.
As a topological space, we have $\Noo = \N \cup \{\infty\}$, and 
this is the notation Clausen and Scholze use for what we call $\Noo$.
In $\lcond$, we have an embedding $\N \emb \Noo$, where
$\N$ denotes the constant sheaf, and an element
$\infty : 1 \to \Noo$. But there is also `more' to $\Noo$ in the following sense.

\begin{lemma}\label{not-jointly-epic}
The two natural maps $\N \emb \Noo$ and $\infty : 1 \to \Noo$
are \emph{not} jointly epic in $\lcond$.
\end{lemma}
That is, the corresponding embedding $\N \sqcup \{\infty\} \emb \Noo$, with the coproduct computed in $\lcond$, is 
not an isomorphism.
\begin{proof}
By \cite[Chapter VI.7]{geology}, the maps are jointly epic if and only if
there is a cover ${\{T_i \to \Noo\}_i}$ of $\Noo$ such that for each $i$, 
the map $T_i \to \Noo$ factors through $\N$ or through $\infty : 1 \to \Noo$.
By compactness, any map $T_i \to \N$ has finite image, so since there
are only finitely many $i$, the $T_i$ can only cover finitely many elements
of $\N$. Hence there is no such cover of $\Noo$.
\end{proof}

We will eventually see that $\Z[\Noo] \in \Ab(\lcond)$ is internally projective
but not invariantly so.
Here $\Z[-]$ denotes the left adjoint to the forgetful functor
$\Ab(\lcond) \to \lcond$; in other words it is the internalisation of the free
abelian group construction to $\lcond$.
We first recall an important lemma used by Clausen and Scholze
in their proof that $\Z[\Noo]$ is internally projective \cite{analytic}. 
Although this result is already established in \cite{analytic}, we present the 
proof in some detail in order to highlight why it does not
translate to the internal logic of $\lcond$.

\begin{lemma}\label{magic-retract}
Let $\iota : S \emb T$ be an embedding of light profinite sets.
If $S$ has a point, then $\iota$ has a retraction $r : T \to S$
with $r \circ \iota = \id_{S}$.
\end{lemma}
\begin{proof}
By assumption, we may write $S$ as the limit of a sequence
$A_0 \leftarrow A_1 \leftarrow \cdots$ of finite sets.
The maps $A_{i+1} \to A_i$ need not be surjective,
but by replacing $A_i$ with the image of the projection $S \to A_i$
if necessary, we may without loss of generality assume that
$A_{i+1} \to A_i$ is surjective.
We may also assume that $A_0$ has a single element
(using the assumption that $S$ has a point to ensure that 
 the maps are still surjective).
Now the problem of finding a retraction $r : T \to S$
is equivalent to finding a sequence of maps $r_i : T \to A_i$
with certain compatibility conditions expressed by the commutativity of the following
square.
\[\begin{tikzcd}[ampersand replacement=\&,cramped]
	S \& {A_{i+1}} \\
	T \& {A_i}
	\arrow[from=1-1, to=1-2]
	\arrow["\iota", hook, from=1-1, to=2-1]
	\arrow[two heads, from=1-2, to=2-2]
	\arrow[dotted, from=2-1, to=1-2]
	\arrow["{r_i}"', from=2-1, to=2-2]
\end{tikzcd}\]
Building the $r_i$ by induction on $i$, the problem amounts to
showing that the square above admits a diagonal filler (which will be $r_{i+1}$).
Now $A_i$ is a finite set and we can argue separately over each fibre of $A_i$,
so without loss of generality $A_i$ consists of a single point.
Then $A_{i+1}$ is a nonempty finite set, and without loss of generality it
contains exactly two elements (the case of one element is trivial, and if $A_{i+1}$ has more
than two elements then we can find a factorisation $A_{i+1} \to \{0,1\} \to 1$
such that
the fibres of $A_{i+1} \to \{0,1\}$ are smaller than $A_{i+1}$, and proceed by induction).

Thus we consider the lifting problem above where $A_i$ consists of a point and $A_{i+1}$ of two points.
The map $S \to A_{i+1}$ exhibits $S$ as a disjoint union of spaces $X$, $Y$, and $\iota$ exhibits $X$ and $Y$
as disjoint closed subsets $X, Y \subseteq T$.
It remains to show that $T$ has a clopen subset $C \subseteq T$ such that
$X \subseteq C$ and $Y \cap C = \emptyset$; this then determines a map $T \to A_{i+1}$ with the desired property.

In order to find $C$, we first write $X$ and $Y$ as intersections $X = \bigcap_i X_i$, $Y = \bigcap_j Y_j$
of basic clopen subsets $X_i$, $Y_j$ of $T$. 
Since the intersection $\bigcap_i X_i \cap \bigcap_j Y_j = X \cap Y$ is empty and $T$ is compact,
we can find a finite set of clopens $X_i$ that have empty intersection with $Y$.
The intersection of these clopens gives a clopen $C$ as desired.
\end{proof}
\begin{remark}\label{invariant-failure}
As we will see, \cref{magic-retract} is used in the proof of the internal projectivity of $\Z[\Noo]$.
In order to prove $\Z[\Noo]$ \emph{invariant}-internally projective, 
one would crucially need the following strengthening:
if $S, T, R$ are light profinite sets with maps $\iota : S \to T$, $p : T \to R$, and
$s : R \to S$ such that $p \circ \iota \circ s = \id_R$ 
(so that $s$ determines a point of $S$ viewed as an object of the slice over $R$),
	then $R$ has a cover $f : \widetilde R \epito R$ such that
	$f^\star \iota : f^\star S \emb f^\star T$ has a retraction \emph{in the slice over} $\widetilde R$.
This strengthening is false, by the following counterexample: consider the
embedding $\Noo \sqcup \{\infty\} \emb \Noo \sqcup \Noo$ viewed in the slice over $\Noo$.
Defining the retraction on the second summand of $\Noo \sqcup \Noo$ amounts to proving
that $\{\infty\} \emb \Noo$ is complemented, which is simply not true, essentially by
\cref{not-jointly-epic}.

On the other hand one could imagine adapting the proof of \cref{magic-retract}
to the internal logic of $\lcond$ in order to prove the (false) strengthening,
and it is instructive to consider where this goes wrong.
It turns out that the internal logic of $\lcond$ permits a great deal of reasoning
about light profinite sets~\cite{ssd}, but importantly it does not validate the law of excluded middle;
$\lcond$ is not a Boolean topos.
This means that one of the first steps, where we reduce to the case of the maps $A_{i+1} \to A_i$
being surjective by replacing $A_i$ with the image of $S \to A_i$,
	does \emph{not} go through; constructively it is not true that
	every subset of a finite set is finite. This is only true if the subset is complemented,
	and there is no reason to expect that the image of $S \to A_i$ is complemented constructively.%
	\footnote{More precisely, the problem is that the \emph{limited principle of omniscience} (LPO),
		an instance of the law of excluded middle, does not hold in $\lcond$.}
\end{remark}

For completeness, we now reproduce Clausen and Scholze's proof that $\Z[\Noo]$ is internally projective,
	with slight modifications to fit the language of this note.
\begin{theorem}[{cf. \cite[Proposition 3.15]{complex}}]\label{noo-proj}
The abelian group object $\Z[\Noo] \in \Ab(\lcond)$ is internally projective.
\end{theorem}
\begin{proof}
Suppose $f : X \epito Y$ is an epimorphism in $\Ab(\lcond)$, $S \in \lcond$ is
an arbitrary object and $y : S^\star \Z[\Noo] \to S^\star Y$ is some map.
By the universal property of $\Z[\Noo]$ (which is stable under pullback)
we can equivalently write $y$ as
$y : S \times \Noo \to Y$, where we conflate the abelian group object $Y$ with
the underlying object of $\lcond$.
Without loss of generality, we may assume $S$ is representable, so
that $S \times \Noo$ is also representable.
Since $f : X \epito Y$ is epic, we can find a cover $p : T \epito S \times \Noo$ with $T$
representable and a commutative square as follows.
\[\begin{tikzcd}[ampersand replacement=\&,cramped]
	T \& X \\
	{S \times \Noo} \& Y
	\arrow["g", from=1-1, to=1-2]
	\arrow["p", two heads, from=1-1, to=2-1]
	\arrow["f", two heads, from=1-2, to=2-2]
	\arrow["y"', from=2-1, to=2-2]
\end{tikzcd}\]

The problem is to show that after passing to a cover of $S$ we can find a map $x : S \times \Noo \to X$
such that $f \circ x = y$.

We claim that, after passing to a cover of $S$ if necessary, we may assume that
the pullback of $p$ along $S \times \N \emb S \times \Noo$ has a section.
First consider, for each $n \in \N$, the pullback of $p$ along
$S \simeq S \times \{n\} \emb S \times \Noo$. Denote the corresponding object over $S$
by $T_n$. Let $\widetilde S$ denote the product of $T_n$ over $n \in \N$ computed in the
slice over $S$. Then $\widetilde S$ is again a light profinite set, and $\widetilde S \to S$
is surjective by countable choice. By construction, the pullback of $p$ along
$\widetilde S \times \N \emb \widetilde S \times \Noo \epito S \times \Noo$ has a section,
finishing the proof of the claim.

Next, given a section $t : S \times \N \emb T \times_{\Noo} \N$, we claim that without loss of generality
we may assume $t$ is an isomorphism. Note that for each $n \in \N$, $T \times_{\Noo} \{n\}$ is a clopen subset
of $T$, and that $t(S \times \{n\})$ is a closed subset of this clopen.
Consider now the intersection $\widetilde T$ of 
\[T \times_{\Noo} (\Noo \setminus \{n\}) \sqcup t(S \times \{n\}) \subseteq T \]
over all $n$. This $\widetilde T$ is again a light profinite set, being a closed subset of $T$,
and the map $\widetilde T \to S \times \Noo$ is again surjective. 
By considering $\widetilde T$ in place of $T$ we get that $t$ becomes an isomorphism as needed.
From now on we simply assume that $t$ is an isomorphism.

Consider now the pullback 
$T_\infty \coloneqq T \times_{\Noo} \{\infty\}$.
Since $\{\infty\} \emb \Noo$ is a closed embedding, so is $\iota : T_\infty \emb T$.
By~\cref{magic-retract},
there exists a retraction $r : T \to T_\infty$.
We may also assume that $p_\infty : T_\infty \epito S$ has a section, $h : S \emb T_\infty$
(by passing to the cover $T_\infty \epito S$ if necessary).

Now define $\widetilde g :  T \to X $ by the following formula using the group structure on $X$.%
\footnote{The only aspect of the abelian group structure we use is the existence of a
	\emph{Malcev operation}. That is, an operation
		$\mu(x, y,z)$ (namely, $\mu(x, y, z) = x - y + z$) such that $\mu(x, z, z) = x$ and
		$\mu(x, x, z) = z$.}
\[
	\widetilde g \coloneqq g - g \circ \iota \circ r + g \circ \iota \circ h \circ p_\infty \circ r		
\]
We claim that $f \circ \widetilde g = y \circ p$ and that $\widetilde g$ extends along $p$.
After this we will be done, since the extension of $\widetilde g$ along $p$ is a map 
$x : S \times \Noo \to X$ with $f \circ x = y$.

We first explain why $f \circ \widetilde g = y \circ p$.
Since $f \circ g = y \circ p$ and $f$ is a group homomorphism, it suffices to show that
$y \circ p \circ \iota \circ r = y \circ p \circ \iota \circ h \circ p_\infty \circ r$.
We have that $p \circ \iota = \iota_S \circ p_\infty$ where $\iota_S : S \emb S \times \Noo$
is the inclusion of $S \times \{\infty\}$. Using this, the claim reduces to
the fact that $h$ is a section to $p_\infty$, i.e.\ $p_\infty \circ h = \id_{S \times \Noo}$.

Now we explain why $\widetilde g$ extends along $p$.
Since epimorphisms in a topos are effective, it suffices to show that for the two projections
$\pi_1, \pi_2 : T \times_{S \times \Noo} T \to T$, we have $\widetilde g \circ \pi_1 = \widetilde g \circ \pi_2$.
Note that $T \times_{S \times \Noo} T$ is a pullback of representables and so is itself representable.

We claim that the evident map
$\alpha : T \sqcup T_\infty \times_S T_\infty \to T \times_{S \times \Noo} T$ is a an epimorphism.
Since surjections of profinite sets are covers in our site, it suffices to check that
this map is a surjection on the level of profinite sets. 
Thus suppose we are given two points $t_1, t_2$ of the profinite set $T$ with $p(t_1) = p(t_2)$.
The second component of $p(t_1)$ is a point of $\Noo$, and we are free to assume that this is either
in $\N$ or $\infty$. In the first case, we have $t_1 = t_2$ as $p$ pulls back to an isomorphism
over $S \times \N$. Thus in this case $(t_1, t_2)$ is hit by the first component
$T$ of $T \sqcup T_\infty \times_S T_\infty$.
In the second case, $(t_1, t_2)$ tautologically lies in $T_\infty \times_S T_\infty$.

Since $\alpha$ is an epimorphism, in order to show that $\widetilde g \circ \pi_1 = \widetilde g \circ \pi_2$, it
suffices to show that $\widetilde g \circ \pi_1 \circ \alpha = \widetilde g \circ \pi_2 \circ \alpha$.
We check this separately on the two summands of $T \sqcup T_\infty \times_S T_\infty$.
On the first summand it is trivial.
Thus we are left to show that the two composites 
\[T_\infty \times_S T_\infty \rightrightarrows T_\infty \xhookrightarrow{\iota} T \xrightarrow{\widetilde g} X\] are equal.
We have $g \circ \iota = g \circ \iota \circ r \circ \iota$ since $r \circ \iota = \id_{T_\infty}$.
Thus $\widetilde g \circ \iota = g \circ \iota \circ h \circ p_\infty \circ r \circ \iota = 
g \circ \iota \circ h \circ p_\infty$.
This evidently factors through $p_\infty : T_\infty \to S$, finishing the proof.
\end{proof}

\begin{remark}
Let us point out what goes wrong in the proof above if one 
tries to show the stronger statement that $\Z[\Noo]$ is invariant-internally projective.
In this case, rather than $X, Y : \Ab(\lcond)$, we would have $X, Y : \Ab(\lcond / S)$.
Thus, in order to use the group operation on $X$ in defining $\widetilde g$, 
we would have to know that the terms going into $\widetilde g$ are maps in the slice $\lcond / S$ over $S$.
The only problem with this is that the retraction $r : T \to T_\infty$
\emph{cannot} in general be taken to respect the maps down to $S$ (cf.\ \cref{invariant-failure}).
\end{remark}

We now turn to the main result of this note.
\begin{theorem}\label{main-result}
The abelian group object $\Noo^\star \Z[\Noo] \in \Ab(\lcond / \Noo)$ is \emph{not}
internally projective. Thus, while $\Z[\Noo] \in \Ab(\lcond)$ is internally projective,
		   it is not invariant-internally projective.
\end{theorem}
\begin{proof}
Consider the closed subset $X$ of $\Noo \times \Noo \times \{0,1\}$ consisting of triples
$(i, j, e)$ such that if either $i \in \N$ or $j \in \N$, then $e = 1$ if and only if $i = j$.
Thus $X$ is a light profinite set and comes with a surjection $\pi : X \epito \Noo \times \Noo$ 
which is an isomorphism away from $(\infty, \infty)$ -- the fibre of $X$ over
$(\infty, \infty)$ consists of two points.
Let $\pi_1 : X \to \Noo$ denote the first projection, so that
$(X, \pi_1)$ defines an object of $\lcond / \Noo$.
Applying free abelian groups to the epimorphism $\pi : (X, \pi_1) \epito \Noo^\star \Noo$ in $\lcond / \Noo$,
we obtain an epimorphism $f : \Z[(X, \pi_1)] \epito \Noo^\star \Z[\Noo]$
in $\Ab(\lcond / \Noo)$.

We now claim that there is \emph{no} cover of $\Noo$ such that
$f$ has a section after pulling back to this cover.
This implies that $\Noo^\star \Z[\Noo]$ is not internally projective.

More precisely, we establish the following claim using the internal language of $\lcond$ in \cref{key-internal}. 
Let $i : I \to \Noo$ be an arbitrary map in $\lcond$, and
suppose that $s : i^\star \Noo^\star \Z[\Noo] \to i^\star \Z[(X, \pi_1)]$ is a section to 
$i^\star f$. The claim is that $i$ factors through
$\N \sqcup \{\infty\} \emb \Noo$, so that $i$ is non-epic by \cref{not-jointly-epic}.
\end{proof}

\begin{lemma}\label{key-internal}
We work in the internal logic of $\lcond$. Let $i \in \Noo$ be arbitrary.
Define $X_i$ to be the set of pairs $(j, e) \in \Noo \times \{0,1\}$
such that if either $i \ne \infty$ or $j \ne \infty$, then $e = 1$ if and only if $i = j$.
Suppose $s : \Z[\Noo] \to \Z[X_i]$ is a section to $\Z[\pi]$ where $\pi : X_i \to \Noo$ is the
first projection. Then $i = \infty$ or $i \ne \infty$.
\end{lemma}

\begin{proof}
We denote by $[x]$ the generator of a free abelian group $\Z[X]$ given by $x \in X$.
Let $a \coloneqq s([i])$ and $b \coloneqq s([\infty])$. We have that $\Z[\pi](a) = [i]$,
	$\Z[\pi](b) = [\infty]$, and $a = b$ if $i = \infty$.
Write $a$ as a linear combination $\sum_k n_k [x_k]$ with $x_k \in X_i$,
so that $\Z[\pi](a) = \sum_k n_k [\pi(x_k)]$.

As we will see in \cref{eq-coprod}, the equation $\sum_k n_k [\pi(x_k)] = [i]$
is witnessed explicitly by a way of collecting like terms in the left hand side.
Importantly, if we have $\pi(x_k) = \pi(x_l)$ in $\Noo$, 
	then we can test whether $x_k = x_l$ in $X_i$;
this holds iff the components of $x_k$ and $x_l$ lying in $\{0,1\}$ are equal.
In the event that we find $x_k \ne x_l$, we must have that these points are
$(\infty, \infty, 0)$ and $(\infty, \infty, 1)$, so that in particular $i = \infty$.
If $i = \infty$ then there is nothing more to prove,
so we may suppose that in collecting like terms,
we in fact have $x_k = x_l$, and we can collect like terms already in $\Z[X_i]$.
That is, we may assume $a = [x]$ with $x \in X_i$ such that $\pi(x) = i$.
Similarly we may assume that $b = [y]$ with $y \in X_i$ such that $\pi(y) = \infty$.

Consider now the second components of $x$ and $y$, lying in $\{0,1\}$. If the second component of $x$
is $0$, then we have $(i, 0) \in X_i$ so that $i = \infty$ and we are done.
Similarly if the second component of $y$ is $1$, then we have $(\infty, 1) \in X_i$
and so $i = \infty$. In the remaining case, we have
$x = (i,1)$ and $y = (\infty, 0)$. In particular $x \ne y$ so that $a \ne b$ and $i \ne \infty$,
	finishing the proof.
\end{proof}

The following lemma should be read constructively since we apply it in the
internal logic of $\lcond$. We use it in the proof of \cref{key-internal}
in the case where $A_x = \Z$ for all $x$ so that $\bigoplus_{x \in X} A_x \cong \Z[X]$.
This result essentially appears in \cite{internal-coprod} and \cite{univalent-cat}
together with the proof sketched below. 
\begin{lemma}\label{eq-coprod}
Suppose we are given a set $X$, a family $A$ of abelian groups over $X$,
		and two elements $\sum_{i \in I} \iota_{x_i}(a_i)$ and $\sum_{j \in J} \iota_{y_j}(b_j)$
		of the coproduct $\bigoplus_{x \in X} A_x$. (Here $I$ and $J$ are finite sets,
				$x_i, y_j \in X$, $a_i \in A_{x_i}$, $b_j \in A_{y_j}$, and $\iota_y$ denotes the
				coproduct coprojection from $A_y$ to $\bigoplus_{x \in X} A_x$).
		Then we have
		\[ \sum_{i \in I} \iota_{x_i}(a_i) = \sum_{j \in J} \iota_{y_j}(b_j) \]
		if and only if this equation can be explained by `collecting like terms'.
		Explicitly, this means we have a third finite set $K$
		with $z_k \in X$ for $k \in K$
		and maps $f : I \to K$, $g : J \to K$ such that
		$x_i = z_{f(i)}$, $y_j = z_{g(j)}$, and
		$\sum_{f(i) = k} a_i = \sum_{g(j) = k} b_j$ in $A_{z_k}$ for all $k \in K$.
\end{lemma}

\begin{proof}[Proof sketch]
We use filtered colimits to describe the coproduct following \cite{internal-coprod,
univalent-cat}. Explictly, let $\C$ denote the category of finite sets $I$ equipped with a map
$x : I \to X$. A map $(I, x) \to (J, y)$ in $\C$ consists of a map $f : I \to J$ such 
that $x = y \circ f$. Because the category of finite sets is finitely cocomplete, so is $\C$,
	 and so in particular $\C$ is filtered. We have $X \cong \colim_{(I, x) \in \C} I$ in the
	 evident way.
	 The coproduct functor preserves this colimit in the sense that
\[
	\bigoplus_{x \in X} A_x \cong \colim_{(I, x) \in \C} \bigoplus_{i \in I} A_{x(i)}.
\]
Since finite coproducts are products and the forgetful functor $U : \Ab \to \Set$ preserves filtered colimits, we have
\[
	U\bigoplus_{x \in X} A_x \cong \colim_{(I, x) \in \C} \prod_{i \in I} UA_{x(i)}.
\]
Thus our two elements of the left-hand side are equal if and only if they are equal viewed as elements
of the right-hand side.
The desired claim follows from the fact that two elements of a filtered colimit of sets are equal if and only if
they map to equal elements under some pair of maps with common codomain.
\end{proof}

\begin{remark}
We had to be a bit more careful than might be expected in the statement 
and proof of \cref{eq-coprod} since we are working
constructively. For example, it does not make sense to talk about the sum
$\sum_{i \in I, x_i = x_{i_0}} a_i$ for $i_0 \in I$ 
unless we have a way of testing equality $x_i = x_{i_0}$
in $X$ (since we need the sum to be finitely indexed). Similarly we cannot 
meaningfully talk about collecting `all possible' like terms.
\end{remark}

\section{Internally projective objects}\label{object}
In this concluding section we give a proof that any internally projective object
of a topos is invariant-internally projective. 
See also \cite{simpson} for another explanation of this fact.

\begin{lemma}
Suppose $P \in \E$ is an internally projective object of a topos $\E$.
Then $P$ is invariant-internally projective.
\end{lemma}
\begin{proof}
Let $S \in \E$ be another object of $\E$, $X, Y : \E / S$ objects in the slice
over $S$, $f : X \epito Y$ an epimorphism in the slice, and $y : S^\star P \to Y$
an arbitrary morphism in the slice.
We have to show that, after pulling back to a cover $T \epito S$, we can lift $y$ along 
$f$.

Consider the following naturality square involving the unit of the adjunction
$S_! \dashv S^\star$.
\[\begin{tikzcd}[ampersand replacement=\&,cramped]
	X \& {S^\star S_! X} \\
	Y \& {S^\star S_! Y}
	\arrow[from=1-1, to=1-2]
	\arrow[from=1-1, to=2-1]
	\arrow["\lrcorner"{anchor=center, pos=0.125}, draw=none, from=1-1, to=2-2]
	\arrow[from=1-2, to=2-2]
	\arrow[from=2-1, to=2-2]
\end{tikzcd}\]
We claim that the above square is cartesian. Since $S_! : \E / S \to S$
creates pullbacks it suffices to show that the square is cartesian after applying
$S_!$. We have $S_! S^\star S_! X \cong S \times S_! X$ and that
$S_! X \to S_! S^\star S_! X$ is naturally a pullback of the diagonal
$S \to S \times S$. In this way the claim follows from pullback pasting and cancellation.

Because $f : X \to Y$ is a pullback of $S^\star S_! f$, the problem of lifting
$y$ against $f$ (up to a cover of $S$) reduces to a lifting problem against
$S^\star S_! f$ (up to a cover of $S$). Because $S_! f : S_! X \to S_! Y$ is an
epimorphism in $\E$, we can conclude by simply appealing to the internal
projectivity of $P$.
\end{proof}

Our reason for presenting the proof abstractly in terms of the adjoints $S_!
\dashv S^\star$ is to make clear what goes wrong if we look at abelian group
objects. In this case, we again have an adjunction $\bigoplus_S \dashv
S^\star$ between $\Ab(\E / S)$ and $\Ab(\E)$. The
problem is that the naturality squares for the unit $\id_{\Ab(\E / S)} \to
S^\star \bigoplus_S$ are generally \emph{not} cartesian.
Indeed, if $f : X \to Y$ is a morphism in $\Ab(\E / S)$,
then the pullback $Y \times_{S^\star \bigoplus_S Y} (S^\star \bigoplus_S X)$
contains a contribution from the direct sum $\bigoplus_S \ker(f)$ of the kernel of $f$.

% ...

\printbibliography
\end{document}